\documentclass[11pt]{amsart}

\usepackage{graphicx,amssymb}
\usepackage{caption,hyperref}
\usepackage{subcaption}
\usepackage{graphicx,tikz,color}
\usepackage{amssymb}
\usepackage{enumerate}
\usepackage[utf8]{inputenc}

\def\RR{\mathbb R}
\def\NN{\mathbb N}
\def\ZZ{\mathbb Z}

\newcommand{\remove}[1]{ }
\newtheorem{definition}{Definition}[section]
\newtheorem{assumption}{Assumption}
\newtheorem{theorem}{Theorem}[section]

\newtheorem{proposition}[theorem]{Proposition}
\newtheorem{lemma}[theorem]{Lemma}
\newtheorem{remark}[theorem]{Remark}

\newtheorem{example}[theorem]{Example}

\numberwithin{equation}{section}

\begin{document}
\title[Internal observability in a triangular domain]{Internal 
observability of the wave equation in a triangular domain}

\author[V. Komornik]{Vilmos Komornik} 
\address{16 rue de Copenhague, 67000 Strasbourg, France}
\email{vilmos.komornik@gmail.com}

\author[A. C. Lai]{Anna Chiara Lai}
\address{Sapienza Università di  Roma,
Dipartimento di Scienze di Base
e Applicate per l'Ingegne\-ria,
via A. Scarpa n. 16,
00161  Roma, Italy}      
\email{anna.lai@sbai.uniroma1.it}
\author[P. Loreti]{Paola Loreti}
\address{Sapienza Università di  Roma,
Dipartimento di Scienze di Base
e Applicate per l'Ingegne\-ria,
via A. Scarpa n. 16,
00161  Roma, Italy}
\email{paola.loreti@sbai.uniroma1.it}
\subjclass[2010]{42B05, 52C20}
\keywords{Internal observability, wave equation, Fourier series, tilings.
}

\maketitle

\begin{abstract}  
We investigate the internal observability of the wave equation with Dirichlet 
boundary conditions in a triangular domain. More precisely, the domain taken 
into exam is the half of the equilateral triangle. Our approach is based on 
Fourier analysis and on tessellation theory: by means of a suitable tiling of 
the rectangle, we extend earlier observability results in the rectangle 
to the case of a triangular domain. The paper includes a general result 
relating problems in general domains to their tiles, and a discussion of the 
triangular case. As an application, we provide an estimation of the observation 
time when the observed domain is composed by three strips with a common 
side to the edges of the triangle.

\end{abstract}

\section{Introduction}\label{s1}

We consider the problem 
\begin{equation}\label{wave}
 \begin{cases}
  u_{tt}-\Delta u=0& \text{in }\RR\times \Omega\\
  u=0&\text{on }\RR\times \partial \Omega\\
  u(t,0)=u_0,~u_t(t,0)=u_1&\text{in }\Omega
 \end{cases}
\end{equation}
where $\Omega$ is a bounded open domain of $\RR^2$ that can be tiled by the open triangle 
$\mathcal T$ whose vertices are $(0,0),(1/\sqrt{3},0)$ and $(0,1)$.
More precisely, we use the symbol $cl(\Omega)$ to denote the closure of a set 
$\Omega$ and we say that an open set $\Omega_1$ tiles  $\Omega_2$ if there 
exist a finite number $N$ of rigid transformations $K_1,\dots,K_N$ such that 
$$cl(\Omega_2)=\bigcup_{h=1}^N K_h cl(\Omega_2).$$
and such that $K_h(\Omega_1)\cap K_j(\Omega_1)=\emptyset$ for all $h\not=j$.
The triangle $\mathcal T$ is the half of an equilateral triangle of side 
$2/\sqrt{3}$, and it tiles the rectangle
$$\mathcal  R:=(0,\sqrt{3})\times (0,1).$$

In particular, we have that the rectangle $\mathcal  R$ can be tiled by 
$\mathcal T$ by means of $6$ rigid transformations 
$K_1,\dots,K_6$: 
  to keep the discussion at an introductory level, we postpone the explicit 
definition of the $K_h$'s to Section \ref{st1}, however such 
tiling is depicted in Figure \ref{tiling}.
\begin{figure}
\begin{tikzpicture}[y=0.80pt, x=0.80pt, yscale=-1.000000, xscale=1.000000, inner sep=0pt, outer sep=0pt]
\begin{scope}[cm={{0.64759,0.0,0.0,-0.63445,(-725.19037,529.07247)}}]
  \begin{scope}
    \path[draw=black,line join=miter,line cap=rect,miter limit=3.25,line
      width=0.800pt] (12.6600,363.8480) -- (215.2190,13.0040) -- (12.6600,13.0040)
      -- (12.6600,363.8480) -- cycle;

    \path[draw=black,line join=miter,line cap=rect,miter limit=3.25,line
      width=0.800pt] (12.6600,363.8480) -- (417.7810,363.8480) --
      (316.5000,188.4260) -- (12.6600,363.8480) -- cycle;

    \path[draw=black,line join=miter,line cap=rect,miter limit=3.25,line
      width=0.800pt] (620.3400,13.0040) -- (215.2190,13.0040) -- (316.5000,188.4260)
      -- (620.3400,13.0040) -- cycle;

    \path[draw=black,line join=miter,line cap=rect,miter limit=3.25,line
      width=0.800pt] (620.3400,13.0040) -- (417.7810,363.8480) --
      (620.3400,363.8480) -- (620.3400,13.0040) -- cycle;

    \path[draw=black,line join=miter,line cap=rect,miter limit=3.25,line
      width=0.800pt] (12.6600,363.8480) -- (215.2190,13.0040) -- (316.5000,188.4260)
      -- (12.6600,363.8480) -- cycle;

    \path[draw=black,line join=miter,line cap=rect,miter limit=3.25,line
      width=0.800pt] (620.3400,13.0040) -- (417.7810,363.8480) --
      (316.5000,188.4260) -- (620.3400,13.0040) -- cycle;

  \end{scope}
\end{scope}
\path[xscale=1.010,yscale=0.990,fill=black,line join=miter,line cap=butt,line
  width=0.800pt] (-687.1593,478.9459) node[above right] (text4187) {$T_1$};

\path[xscale=1.010,yscale=0.990,fill=black,line join=miter,line cap=butt,line
  width=0.800pt] (-609.4599,422.4844) node[above right] (text4191) {$K_2(T_1)$};

\path[xscale=1.010,yscale=0.990,fill=black,line join=miter,line cap=butt,line
  width=0.800pt] (-489.8859,422.7302) node[above right] (text4191-3)
  {$K_5(T_1)$};

\path[xscale=1.010,yscale=0.990,fill=black,line join=miter,line cap=butt,line
  width=0.800pt] (-585.0149,346.5223) node[above right] (text4191-7)
  {$K_3(T_1)$};

\path[xscale=1.010,yscale=0.990,fill=black,line join=miter,line cap=butt,line
  width=0.800pt] (-536.5394,478.0724) node[above right] (text4191-35)
  {$K_4(T_1)$};

\path[xscale=1.010,yscale=0.990,fill=black,line join=miter,line cap=butt,line
  width=0.800pt] (-418.6524,346.5535) node[above right] (text4191-9)
  {$K_6(T_1)$};

\end{tikzpicture}
 \caption{The tiling of $\mathcal R$ with $\mathcal T$.  Note that $K_1$ is the identity map, hence $K_1(\mathcal T)=\mathcal T$. \label{tiling}}\end{figure}
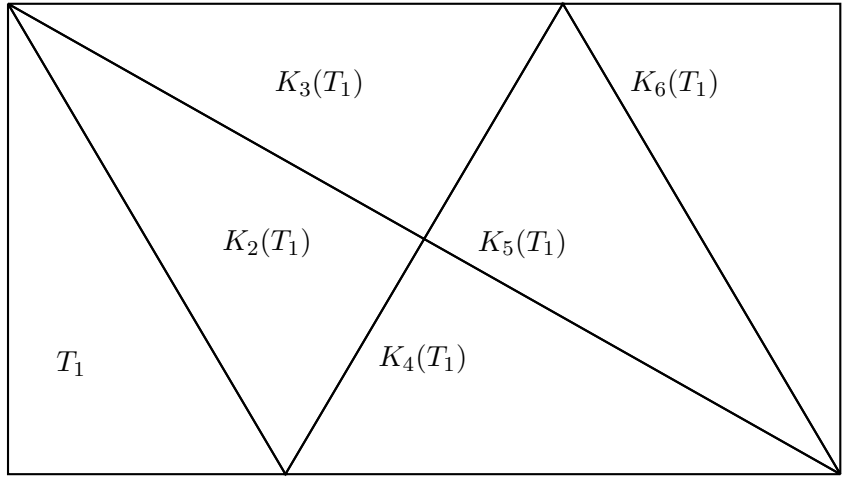
 
 As it is well known, a complete orthonormal base for $L^2( \mathcal R)$ is given by the eigenfunctions of $-\Delta$ in $H_0^1( \mathcal R)$
 $$\overline e_{k}:=\sin(\pi k_1 x_1/\sqrt{3})\sin (\pi k_2 x_2), \quad \text{where } k=(k_1,k_2),~k_1,k_2\in \NN$$
 and the associated eigenvalues are $\gamma_k=\frac{k_1^2}{3}+k_2^2$. 
 In \cite{triangle}, a folding technique (that we recall in detail in Section 
\ref{st1}) is used to derive from $\{\overline e_{k}\}$ an orthogonal base 
$\{e_{k}\}$ of $L^2(\mathcal T)$ formed by the eigenfunctions of $-\Delta$ in 
$H_0^1(\mathcal T)$. In particular, $\{e_{k}\}\subset span\{\overline e_{k}\}$ 
and $\{e_k\}$ and $\{\overline e_k\}$ share the same eigenvalues $\gamma_k$.
 
  The explicit knowledge of a eigenspace for $H_0^1(\mathcal T)$ allows us 
  to set the problem \eqref{wave} (with $\Omega=\mathcal T$) in the framework 
of Fourier analysis. Our goal is to exploit the deep relation between the 
eigenfunctions for $H_0^1(\mathcal R)$ and those of $H_0^1(\mathcal T)$ in order 
to extend known observability results for $\mathcal R$ to $\mathcal T$.
  
In particular, we are interested in the \emph{internal observability} of 
\eqref{wave}, i.e., in the validity of the estimates 
$$\|u_0\|_{L^2(\Omega)}^2+\|u_1\|^2_{H^{-1}(\Omega)}\asymp 
\int_0^T\int_{\Omega_0} |u(t,x)|^2 dx $$
where $\Omega_0$ is a subset of $\Omega$ and $T$ is sufficiently large. Here 
and in the sequel $A\asymp B$ means $c_1 A\leq B \leq c_2 A$ with some 
constants 
$c_1$ and $c_2$ which are independent from $A$ and $B$. When we need to stress 
the dependence of these estimates on the couple of constants $c=(c_1,c_2)$, we 
write $A \asymp_{c} B$. Also by writing $A\leq_c B$ we mean
the inequality $ c A\leq B$ while the expression $A\geq_c B$ denotes $c A\geq 
B$.

  \subsection{Statement of the main results}
  We begin by introducing a few notations.  Let 
$\{e_k\}\subset H^1_0(\Omega)\cap L^2(\Omega)$ be an orthonormal base of  
$L^2(\Omega)$ 
formed by eigenvalues of $-\Delta$ and let $\{\gamma_k\}$ be the associated 
eigenvalues. 
Denote by $D^s(\Omega)$ the completion of $\{e_k\}$ with respect to the 
Euclidean norm
$$\left\| \sum_{k\in\ZZ^2} c_k \gamma_k\right\|_s:=\left(\sum_{k\in\ZZ^2} 
\gamma_k^s |c_k|^2\right)^{1/2}.$$
Identifying $L^2(\Omega)$ with its dual we have
$$D^0(\Omega)=L^2(\Omega),\quad D^{-1}(\Omega)= H^{-1}(\Omega).$$

 We have
  \begin{theorem}\label{thmmainintro}
 Let  $\overline u$ be the solution of
  \begin{equation}\label{wavegeneralrectintro}
 \begin{cases}
  \overline u_{tt}-\Delta \overline u=0& \text{on }\RR\times  {\mathcal R}\\
  \overline u=0&\text{in }\RR\times \partial \mathcal R\\
  \overline u(t,0)= \overline u_0,~u_t(t,0)=\overline u_1&\text{in } {\mathcal R},
 \end{cases}
\end{equation}
 let $\overline S$ be a subset of $\mathcal R$ and assume that there exists a 
constant $T_{\overline S}\geq 0$ such that if $T>T_{\overline S}$ then there 
exists a couple of constants $c=(c_1,c_2)$ such that  $\overline u$ satisfies
 \begin{equation}\label{obsrectanglegintro}
 \|\overline u_0\|_0^2+\|\overline u_1\|^2_{-1}\asymp_{c} 
\int_0^T\int_{\overline 
S} |\overline u(t,x)|^2 dx  
 \end{equation}
  for all $(\overline u_0, \overline u_1)\in D^0(\mathcal R)\times 
D^{-1}(\mathcal R)$. 
Moreover let $u$ be the solution $u$ of 
  \begin{equation}\label{wavegeneraltrintro}
 \begin{cases}
  u_{tt}-\Delta u=0& \text{on }\RR\times  {\mathcal T}\\
  u=0&\text{in }\RR\times \partial \mathcal T\\
  u(t,0)=u_0,~u_t(t,0)=u_1&\text{in } {\mathcal T}.
 \end{cases}
\end{equation}
and set
$$S:=\bigcup_{h=1}^6 (K^{-1}_h \overline S\cap \mathcal T)$$
Then for each $T>T_{\overline S}$ the solution $u$  satisfies
\begin{equation}\label{obstrgintro}
\|u_0\|_0^2+\|u_1\|^2_{-1}\asymp_{c} \int_0^T\int_{S} |u(t,x)|^2 dx. 
\end{equation}
for all $(u_0,u_1)\in D^0(\mathcal T)\times D^{-1}(\mathcal T)$.

The result also holds by replacing every occurrence of $\asymp_c$ with $\leq_c$ 
or $\geq_c$.
\end{theorem}
We point out that the time of observability $T_{\bar S}$ stated in Theorem 
\ref{thmmainintro}, as well as the couple $c$ of constants in the estimates 
\eqref{obsrectanglegintro} and \eqref{obstrgintro}, are the same for both the 
domains $\mathcal R$ and $\mathcal T$. Also note that in Section \ref{st1} we 
prove a slightly stronger version of Theorem \ref{thmmainintro}, that is 
Theorem \ref{thmmain}: its precise statement requires some technicalities that 
we chose to avoid here, however we may anticipate to the reader that 
the assumption on initial data $(\overline u_0,\overline u_1)\in  D^0(\mathcal 
R)\times D^{-1}(\mathcal R)$ can be weakened by replacing $ D^0(\mathcal 
R)\times D^{-1}(\mathcal R)$ with an appropriate subspace. 

Now we state the second main result of the present paper: its proof strongly 
relies on Theorem \ref{thmmainintro} and on a couple of technical lemmas 
that can be found in Section \ref{st2}. 
 \begin{figure}
\includegraphics[scale=0.28]{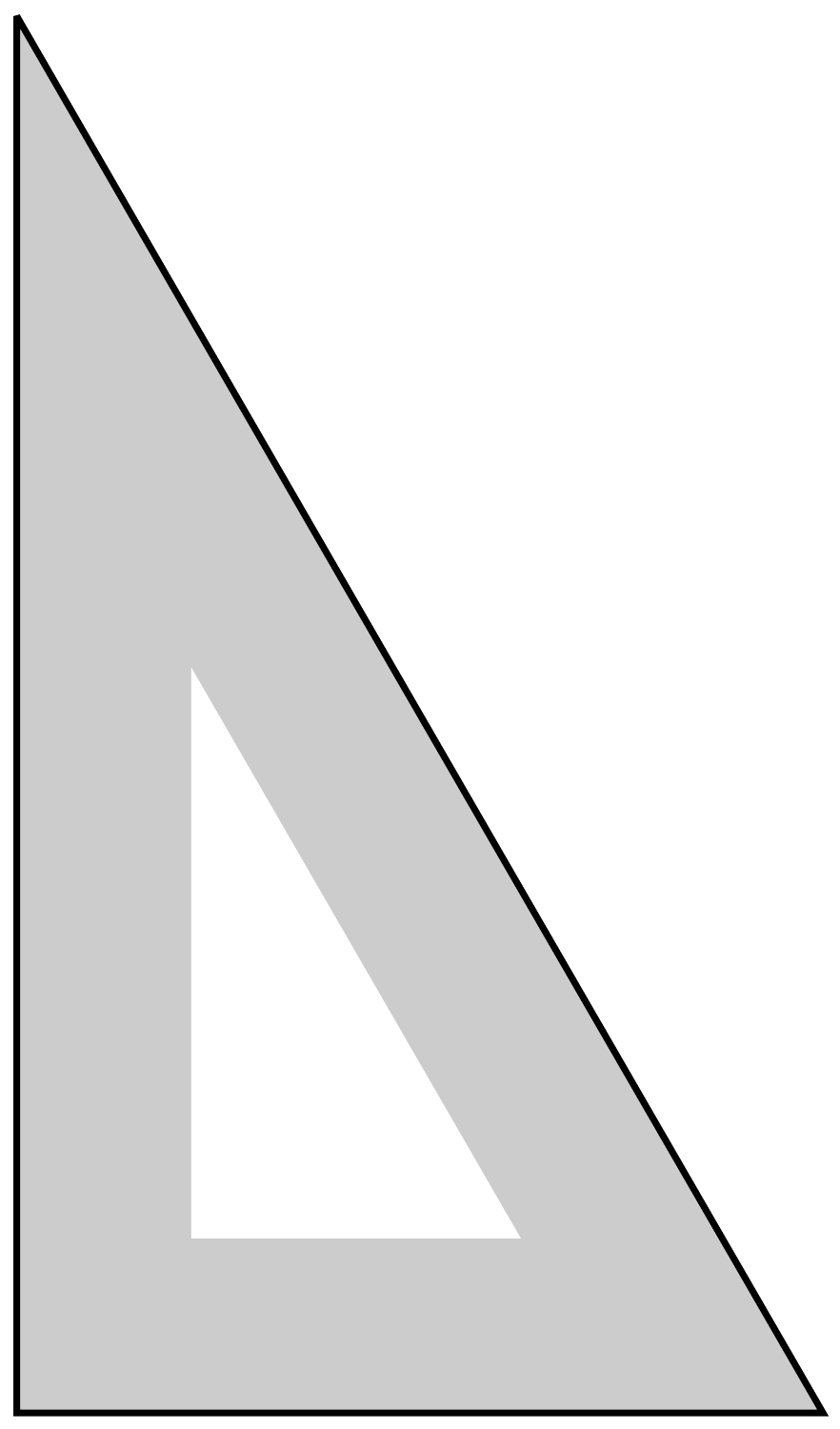}\caption{The triangle $\mathcal T$ and, 
in gray, the observation domain $S_\alpha$ with 
$\alpha=0.125$.\label{stripsfigintro}}
  \end{figure}
\begin{theorem}[Observability on strips along the edges of $\mathcal 
T$]\label{stripsedgeintro}
 Let $\alpha\in(0,1/(3+\sqrt{3})]$. Set $r_\alpha:=1-\alpha(3+\sqrt{3}),$ 
 $$S_{\alpha}:=  {\mathcal T}\setminus cl(r_\alpha \mathcal T+ (\alpha,\alpha)) 
             ,$$
 $$t_\alpha:=\inf_{k\in \NN} \int_0^\alpha \sin^2(\pi k x/\sqrt{3})dx$$
 and
$$T_\alpha:=8\sqrt{\frac{5}{\sqrt{3}}t_\alpha}.$$
 If $u$ is the solution of \eqref{wavegeneraltrintro}, then for every 
$T>T_\alpha$
 \begin{equation}\label{obstripsintro}
\|u_0\|_0^2+\|u_1\|^2_{-1}\leq_{c_\alpha} \int_0^T\int_{S_\alpha} |u(t,x)|^2 
dx. 
\end{equation}
for all $(u_0,u_1)\in D^0(\mathcal T)\times D^{-1}(\mathcal T)$ with
$$c_\alpha:= \frac{T}{\pi}\left(\frac{t_\alpha}{\sqrt{3}}
-\frac{40}{3 T^2}\right)>0.$$
\end{theorem}
  \begin{remark}
  If $\alpha\leq 1/(3+\sqrt{3})\simeq 0.211$ then $S_\alpha$ can be 
equivalently viewed as the intersection between $\mathcal T$ and the union of 
three open strips $s_1(\alpha)$, $s_2(\alpha)$ and $s_3(\alpha)$ of width equal 
to $\alpha$, each of which has a common side with an edge of $\mathcal T$, see 
Figure \ref{stripsfigintro}. If $\alpha=1/(3+\sqrt{3})$ then we are setting as 
domain of observation $\mathcal T\setminus \{(\alpha,\alpha)\}$:  note that in 
this case the point  $(\alpha,\alpha)$ is the incenter of $\mathcal T$. Finally 
 if $\alpha\geq 1/(3+\sqrt{3})$ then the union of $s_1(\alpha),s_2(\alpha)$ and 
$s_3(\alpha)$ covers $\mathcal T$.
  \end{remark}
\begin{proof}
Let 
$$\bar S_\alpha:=[(0,\alpha)\times(0,1)]\cup[(0,\sqrt{3})\times (0,\alpha)].$$
Using a result in 
\cite{KomMia2013}, we show in Lemma \ref{lobs1} that if $T>T_\alpha$ then the 
solution 
$\overline u$ of \eqref{wavegeneralrectintro} satisfies 
\eqref{obsrectanglegintro} for all initial data $(\overline u_0, \overline 
u_1)\in D^0(\mathcal R)\times D^{-1}(\mathcal R)$. Then, by 
Theorem 
\ref{thmmainintro}, setting 
$$S'_\alpha:= \bigcup_{h=1}^6 (K^{-1}_h \overline S_\alpha \cap \mathcal T)$$
we have that if $u$ is the solution of \eqref{wavegeneraltrintro} then 
$T>T_\alpha$ implies 
$$\|u_0\|_0^2+\|u_1\|^2_{-1}\leq c_\alpha \int_0^T\int_{S'_\alpha} |u(t,x)|^2 
dx.$$
for all  $(u_0,u_1)\in D^0(\mathcal T)\times D^{-1}(\mathcal T)$. The claim 
hence 
follows by showing that $S_\alpha=S'_\alpha$ for all $\alpha\geq0$, which 
is proved in Lemma \ref{lobs2}.
\end{proof}
\subsection{Organization of the paper.} In 
Section \ref{sgeneral} we consider a 
generic domain $\Omega_1$ tiling a larger domain $\Omega_2$:  we establish a 
result, Theorem \ref{thmprol}, relating the observability properties of wave 
equation on $\Omega_2$ and on its tile $\Omega_1$. In Section \ref{st1} we 
specialize this result to the case in which  $\Omega_1$ is the triangle 
$\mathcal T$ and the tiled domain $\Omega_2$ is the rectangle $\mathcal R$: 
this 
is the core of the proof of Theorem \ref{thmmainintro}. Finally Section 
\ref{st2} is devoted to the proof of Theorem \ref{stripsedgeintro}.

\section{An observability result on tilings}\label{sgeneral}
The goal of this section is to state an equivalence between an observability 
problem on a domain $\Omega_1$ and an observability problem on a larger domain 
$\Omega_2$, under the assumption that $\Omega_1$ tiles $\Omega_2$. 
We begin with some definitions.
\begin{definition}[Tilings, foldings and prolongations]
 Let $\Omega_1$ and $\Omega_2$ be two open bounded subsets of $\RR^2$. We say 
that $\Omega_1$ \emph{tiles} $\Omega_2$ if there 
 exists a set $\{K_h\}_{h=1}^N$ rigid transformations of $\RR^2$ 
  such that
 $$cl(\Omega_2)=\bigcup_{h=1}^N K_hcl(\Omega_1).$$
and such that $K_h(\Omega_1)\cap K_j(\Omega_1)=\emptyset$ for all $h\not=j$.
 
 Let $(\Omega_1,\{K_h\}_{h=1}^N)$ be a tiling of $\Omega_2$ and 
$\delta=(\delta_1,\dots,\delta_N)\in\{-1,1\}^N$. The \emph{prolongation with 
coefficients $\delta$} of a function $u:\Omega_1\to \RR$ to $\Omega_2$ is the 
function $\mathcal P_\delta u:\Omega_2\to \RR$ 
  $$\mathcal P_\delta u(K_h x)=\delta_h u(x) \text{ for each }h=1,\dots,N.$$
 The \emph{folding with coefficients $\delta$} of a function $\overline 
u:\Omega_2\to \RR$ is the function $\mathcal F_\delta \overline u:\Omega_1\to 
\RR$ 
  $$\mathcal F_\delta \overline u(x)=\frac{1}{N^2}\sum_{h=1}^N \delta_h 
\overline u(K_h x) \text{ for each }h=1,\dots,N.$$

 A tiling $(\Omega_1,K_h)$ of $\Omega_2$ is \emph{admissible} if there exists 
$\delta\in\{-1,1\}^N$ such that 
  \begin{equation}\label{admdef}
\mathcal F_\delta\varphi \in H^1_0(\Omega_1) \quad \forall \varphi \in 
H_0^1(\Omega_2).
   \end{equation}
\end{definition}
under scripts and we simply write $\mathcal P$ and $\mathcal F$.

 \begin{example}\label{exadm}
  We show in Lemma \ref{ladm} below that the tiling of $\mathcal R$ with 
$\mathcal T$ depicted in Figure \ref{tiling} is admissible, in particular 
\eqref{admdef} holds with $\delta=(1,-1,1,1,-1,1)$.
    
  On the other hand the tiling of $\mathcal R':=(0,1/\sqrt{3})\times(0,1)$ 
given 
by the transformations $K_1':=id$ and $$K_2':(x_1,x_2)\mapsto 
-(x_1,x_2)+(1/\sqrt{3},1),$$
  see Figure \ref{tilingbad}, is not admissible. Let indeed  
$v_1:=(1/\sqrt{3},0)$, 
  $v_2:=(0,1)$ and $x_\lambda:=\lambda v_1 +(1-\lambda)v_2$ with $\lambda\in 
(0,1)$. 
  Then $x_\lambda\in \partial \mathcal T$ and 
  $$K_2(x_{\lambda})=x_{1-\lambda}$$
  Therefore it suffices to choose $\varphi\in H_0^1(\mathcal R)$ such that 
$\varphi(x_\lambda)\not=\pm \varphi(x_{1-\lambda})$ to obtain 
  $$\mathcal F_\delta \varphi(x_\lambda)=\delta_1 \varphi(x_\lambda)+\delta_2 
\varphi(x_{1-\lambda})\not=0$$
  for all $\delta_1,\delta_2\in\{-1,1\}$. Consequently $\mathcal F_\delta 
\varphi \not\in H_0^1(\mathcal T)$ for all $\delta\in\{-1,1\}^2$.
 \end{example}

 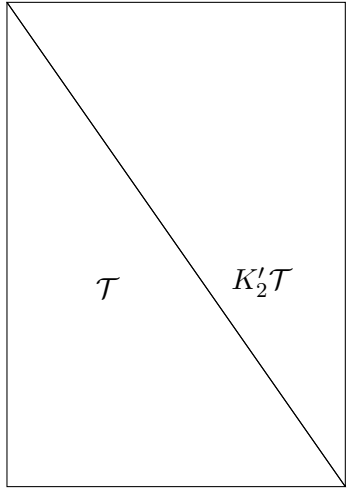
\begin{figure}
 \definecolor{cffffff}{RGB}{255,255,255}

\begin{tikzpicture}[y=0.80pt, x=0.80pt, yscale=-1.000000, xscale=1.000000, inner sep=0pt, outer sep=0pt]
\begin{scope}[cm={{0.65622,0.0,0.0,-0.54271,(349.54106,742.37215)}}]
  \begin{scope}
    \path[fill=cffffff,nonzero rule] (0.0000,0.0000) -- (254.0000,0.0000) --
      (254.0000,432.0000) -- (0.0000,432.0000) -- (0.0000,0.0000) -- cycle;

    \path[draw=cffffff,line join=miter,line cap=butt,miter limit=10.00,line
      width=0.024pt] (0.0000,432.0000) -- (254.0000,432.0000) -- (254.0000,0.0000)
      -- (0.0000,0.0000) -- (0.0000,432.0000) -- cycle;

    \path[draw=black,line join=miter,line cap=rect,miter limit=3.25,line
      width=0.400pt] (5.0780,427.1250) -- (248.8630,4.8750) -- (5.0780,4.8750) --
      (5.0780,427.1250) -- cycle;

    \path[draw=black,line join=miter,line cap=rect,miter limit=3.25,line
      width=0.400pt] (248.8630,4.8750) -- (5.0780,427.1250) -- (248.8630,427.1250)
      -- (248.8630,4.8750) -- cycle;

  \end{scope}
\end{scope}
\path[xscale=1.140,yscale=0.877,fill=black,line join=miter,line cap=butt,line
  width=0.800pt] (346.3738,741.5263) node[above right] (text4200) {$\mathcal
  T$};

\path[xscale=1.140,yscale=0.877,fill=black,line join=miter,line cap=butt,line
  width=0.800pt] (402.7135,741.2077) node[above right] (text4200-3)
  {$K_2'\mathcal T$};

\end{tikzpicture}
 \caption{\label{tilingbad} A non-admissible tiling of $\mathcal 
R'=(0,1/\sqrt{3})\times(0,1)$ with $\mathcal T$. }
\end{figure}
 \begin{remark}
  We borrowed the notion of prolongation and folding from \cite{triangle}: 
while our definition of $\mathcal P_\delta$ is exactly as it is given in 
\cite{triangle}, we introduced a normalizing term $1/N^2$ in the definition of 
$\mathcal F_\delta$ in order to enlighten the notations. Note that the 
following 
equality holds:
 \begin{equation}\label{fold2}
 \mathcal F_\delta(\mathcal P_\delta u)=\frac{1}{N}u  
 \end{equation}
  for all $u:\Omega_1\to\RR$.

  Also remark that we shall need to prolong and fold also functions 
$u:\RR\times \Omega_1\to \RR$ and 
 $\bar u:\RR\times \Omega_2\to \RR$, in this case the definition of $\mathcal 
P$ and $\mathcal F$ naturally extends by applying the transformations $K_h$'s 
to the spatial variables $x$. For instance
 if  $u:\RR\times \Omega_1\to \RR$ then its prolongation to $\RR\times 
\Omega_2$ reads
 $$\mathcal P_\delta u(t,K_h x)=\delta_h u(t,x).$$ 
 \end{remark}
%
%
%

We want to establish a relation between solutions of a wave equation with 
Dirichlet boundary conditions and their prolongation. To this end we introduce 
the notations
$$\mathcal P_\delta L^2(\Omega_1):=\{\mathcal P_\delta u\mid u\in 
L^2(\Omega_1)\},$$
$$\mathcal P_\delta H^1_0(\Omega_1):=\{\mathcal P_\delta u\mid u\in 
H^1_0(\Omega_1)\}$$
and
$$\mathcal P_\delta H^{-1}(\Omega_1):=\{\mathcal P_\delta u\mid u\in 
H^{-1}(\Omega_1)\}.$$

Note that $\mathcal P_\delta L^2(\Omega_1)\subset L^2(\Omega_2)$, $\mathcal 
P_\delta H^1_0(\Omega_1)\subset H^1_0(\Omega_2)$ and $\mathcal P_\delta 
H^{-1}(\Omega_1)\subset H^{-1}(\Omega_1)$ .

All results below hold under the following assumptions on the domains 
$\Omega_1$, $\Omega_2$ and on a base $\{e_k\}$ for $L^2(\Omega_1)$:
\begin{assumption}\label{A1}
$(\Omega_1,\{K_h\}_{h=1}^N)$ is an admissible tiling of $\Omega_2$.
\end{assumption}

\begin{assumption}\label{A2}
$\{e_k\}$ is a base of eigenvectors of $-\Delta$ in $H^1_0(\Omega_1)$, it is 
defined on $\Omega_1\cup \Omega_2$ and there exists $\delta\in\{-1,1\}^N$ such 
that
 $$\mathcal P_\delta (e_k|_{\Omega_1})=e_k|_{\Omega_2}$$ 
 for each $k\in\NN$.
\end{assumption}
\begin{remark}[Some remarks on Assumption \ref{A2}]
 We note that Assumption \ref{A2} can be equivalently stated as
\begin{equation}\label{opdef}
e_k(K_h x)=\delta_h e_k(x)\quad \text{
for all $x\in \Omega_1$, $h=1,\dots,N$, $k\in\NN$.} 
\end{equation}
Indeed,  by definition of prolongation and noting $\delta_h^2\equiv 1$, we have
\begin{equation*}
e_k(K_h x)=\delta_h^2 e_k(K_h x)=\delta_h \mathcal P_\delta e_k(x)=\delta_h 
e_k(x).
\end{equation*}
 for every $x\in \Omega_1$, $h=1,\dots,N$ and $k\in\NN$. 
 \\ Also remark that, in view of \eqref{fold2}, Assumption \ref{A2} also 
implies 
\begin{equation}\label{fold}
\mathcal F_\delta e_k= \frac{1}{N} e_k. 
\end{equation}
\end{remark}
\begin{example}
 Let $\Omega_1=(0,\pi)^2$ and $\Omega_2=(0,2\pi)^2$. Consider the 
transformations of $\RR^2$
$$ \begin{array}{ll}
 K_1:= \text{id}, &~K_2: (x_1,x_2)\mapsto (-x_1+2\pi,x_2),\\
 K_3:(x_1,x_2)\mapsto (x_1,-x_2+2\pi),&~K_4: (x_1,x_2)\mapsto 
-(x_1,x_2)+(2\pi,2\pi)  
 \end{array}.$$
  Then $\{\Omega_1,\{K_h\}_{h=1}^4\}$ is a tiling for $\Omega_2$.  In 
particular, Assumption \ref{A1} is satisfied: indeed setting 
$\delta=(1,-1,-1,1)$ we have for each $\varphi\in H_0^1(\Omega_2)$
  $$\mathcal F_\delta \varphi(x)= 0 \quad \forall x\in \partial \Omega_1.$$
   Also note that the functions
 $$e_k(x):=\sin (k_1 x_1)\sin (k_2x_2)\quad k=(k_1,k_2)\in\NN^2$$
 satisfy Assumption \ref{A2}, indeed they are a base for $L^2(\Omega_1)$ 
composed by eigenfunctions of $-\Delta$ in $H_0^1(\Omega_1)$ and
 $$e_k( K_h x):=\delta_h e_k(x)$$
 for all $x\in \RR^2$, $h=1,\dots,4$ and $k\in \NN^2$. The space 
$\mathcal P_\delta L^2(\Omega_1)$ in this case coincides with the space of 
so-called \emph{$(2,2)$-cyclic functions}, i.e.,  functions in $L^2(\Omega_2)$ 
which are odd with respect to both axes $x_1=\pi$ and $x_2=\pi$.
 We refer to \cite{KomLor159} for some results on observability of wave 
equation 
with $(p,q)$-cyclic initial data. 
\end{example}

\vskip0.5cm
Our starting point is to show that, under Assumptions \ref{A1} and \ref{A2}, 
the base of eigenfunctions $\{e_k\}$ is also a base of eigenfunctions also for 
an appropriate subspace of $L^2(\Omega_2)$, and to compute the associated 
coefficients.  
\begin{lemma}\label{lcr}
 Let $\Omega_1,\Omega_2$ and $\{e_k\}$ satisfy Assumption \ref{A1} and 
Assumption \ref{A2}.
 
%

 Then $\{e_k\}\subset H^1_0(\Omega_2)$ and it is also a complete base for 
$\mathcal P_\delta L^2(\Omega_1)$ formed by eigenfunctions of $-\Delta$ in 
$\mathcal P_\delta H_0^1(\Omega)$.
 
 In particular, for every $k\in \NN$, if $u_k$ is the coefficient of 
 $u\in L^2(\Omega_1)$ (with respect to $e_k$) then $Nu_k$ is the coefficient of 
$\mathcal P_\delta u$. 
 \end{lemma}
\begin{proof} The proof is organized two steps. \\
\emph{Claim 1: $\{e_k\}$ is a set of eigenfunctions of $-\Delta$ in 
$H_0^1(\Omega_2)$}. Extending a result given in \cite{triangle}, we need to 
show that, under Assumption \ref{A1} and Assumption \ref{A2}, if $e_k\in 
H_0^1(\Omega_1)$ is a solution of the boundary value problem 
$$ \int_{\Omega_1} \nabla e_k\nabla \varphi dx =\int_{\Omega_1} \gamma_k e_k 
\varphi dx \quad \forall \varphi \in H_0^1(\Omega_1)
$$
for some $\gamma_k\in\RR$, then $e_k$ is also solution of the boundary value 
problem on $\Omega_2$
$$ \int_{\Omega_2} \nabla  e_k \nabla \varphi dx =\int_{\Omega_2} \gamma_k e_k 
\varphi dx \quad \forall \varphi \in H_0^1(\Omega_2).
$$

Now, recall from  Assumption \ref{A1} that if $\varphi \in H_1^0(\Omega_2)$ 
then 
$\mathcal F_\delta \varphi\in H_1^0(\Omega_1)$.  Then it follows again from 
Assumption \ref{A1} and from Assumption \ref{A2} (in particular by recalling 
that $K_h$'s are isometries and \eqref{opdef}) that for all $\varphi \in 
H_0^1(\Omega_2)$
\begin{align*}
 \int_{\Omega_2} \nabla  e_k(x)&\nabla \varphi(x) dx= \int_{\bigcup_{h=1}^N K_h 
\Omega_1} 
 \nabla  e_k(x) \nabla \varphi(x) dx\\
&=\sum_{h=1}^N \int_{\Omega_1} \nabla  e_k(K_h x) \nabla \varphi(K_h x) dx = 
\int_{\Omega_1}  \nabla  e_k(x) \sum_{h=1}^N \delta_h \nabla \varphi(K_h x) dx\\
&= \int_{\Omega_1}\nabla e_k(x) \nabla \mathcal F_\delta \varphi(x) dx= 
\int_{\Omega_1} \gamma_k  e_k(x) \mathcal F_\delta \varphi(x) dx\\
&= \int_{\Omega_2} \gamma_k  e_k(x) \varphi(x) dx.
\end{align*}
and this completes the proof of Claim 1. 

\emph{Claim 2: completeness of $\{e_k\}$ and computation of coefficients} By 
Assumption \ref{A1} and Assumption \ref{A2} and by recalling $\delta_h^2=1$ for 
each $h=1,\dots,N$, we have
\begin{align*}
\int_{\Omega_2} \mathcal P_\delta u(x) e_k(x)dx&=\int_{\Omega_2} \mathcal 
P_\delta u(x) \mathcal P_\delta e_k(x)dx\\
&=
\sum_{h=1}^N\int_{K_h \Omega_1} \mathcal P_\delta u(x) \mathcal Pe_k(x)dx\\
&=\sum_{h=1}^N\int_{K_h \Omega_1} \delta_h^2 u(K_h x) e_k(K_h x)dx\\
&=\sum_{h=1}^N\int_{\Omega_1}  u(x) e_k(x)dx=N\int_{\Omega_1}  u(x) e_k(x)dx,
\end{align*}
where the second to last equality holds because $K_h$'s are rigid 
transformations. Then we may deduce two facts: first if $\{u_k\}$ are the 
coefficients of $u\in L^2(\Omega_1)$ then $\{N u_k\}$ are coefficients of 
$\mathcal P_\delta u$. Secondly, $\{e_k\}$ is a complete base for $\mathcal 
P_\delta L^2(\Omega_1)$, indeed if the coefficients of $\mathcal P_\delta u$ 
are 
identically null, then also the coefficients of $u$ are identically null: since 
$\{e_k\}$ is complete for $\Omega_1$ then $u\equiv 0$ and, consequently, 
$\mathcal P_\delta u\equiv 0$, as well. 
$\Omega_2$ and, 
consequently, $\partial \Omega_2\subset \bigcup_{h=1}^N K_h(\partial 
\Omega_1)$.
\end{proof}
Next result establishes a relation between solutions of wave equations on tiles 
and their prolongations.
\begin{lemma}\label{lprol} 
 Let $\Omega_1,\Omega_2$ and $\{e_k\}$ satisfy Assumption \ref{A1} and 
Assumption \ref{A2}. 
 Let $u$ be the solution of
 \begin{equation}\label{wavegeneral}
 \begin{cases}
  u_{tt}-\Delta u=0& \text{in }\RR\times \Omega_1\\
  u=0&\text{on }\RR\times \partial \Omega_1\\
  u(t,0)=u_0,~u_t(t,0)=u_1&\text{in }\Omega_1
 \end{cases}
\end{equation}
Then $u$ is well defined in $\Omega_1\cup\Omega_2$ and $\overline u=N 
u|_{\Omega_2}$ is the solution of 
  \begin{equation}\label{wavegeneral2}
 \begin{cases}
  \overline u_{tt}-\Delta \overline u=0& \text{on }\RR\times \Omega_2\\
  \overline u=0&\text{in }\RR\times \partial \Omega_2\\
  \overline u(t,0)=\mathcal P_\delta u_0,~\overline u(t,0)=\mathcal P_\delta 
u_1&\text{in }\Omega_2
 \end{cases}
\end{equation}
Conversely, if $\bar u$ is the solution of \eqref{wavegeneral2} then $\mathcal 
F_\delta \bar u$ is the solution of  \eqref{wavegeneral} and for every 
$h=1,\dots,N$ 
\begin{equation}\label{sol}
\mathcal F_\delta \bar u(t,x)= \frac{\delta_h}{N} \bar u(t,K_h x)\quad 
\text{for each $x\in \Omega_1$}. 
\end{equation}

\end{lemma}

\begin{proof}
 Let $\{\gamma_k\}$ be the sequence of eigenvalues associated to $\{e_k\}$ and 
set $\omega_k=\sqrt{\gamma_k}$, for every $k\in\NN$. Expanding $u(t,x)$ with 
respect to $e_k$ we obtain 
 $$u(t,x)=\sum_{k=1}^\infty (a_k e^{i\omega_k t}+b_k e^{-i\omega_k t})e_k(x)$$
with $a_k$ and $b_k$  depending only the coefficients $c_k$ and $d_k$ of 
$u_0$ and $u_1$ with respect to $\{e_k\}$.  In particular $a_k+b_k=c_k$ and 
$a_k-b_k=-id_k/\omega_k$.
We then have that the natural domain of $u$ coincides with the one of 
$\{e_k\}$'s, hence it is included in $\Omega_1\cup\Omega_2$. 
 By Lemma \ref{lcr}  the coefficients of $\mathcal P_\delta u_0$ and $\mathcal 
P_\delta u_1$ are $Nc_k$ and
$N d_k$, respectively. Then it is immediate to verify that
 $$N u(t,x)=\sum_{k=1}^\infty (N a_k e^{i\omega_k t}+ N b_k e^{-i\omega_k 
t})e_k(x)$$
is the solution of \eqref{wavegeneral2}. 

Now, let 
$$\bar u(t,x)= \sum_{k=1}^\infty (\bar a_k e^{i\omega_k t}+ \bar b_k 
e^{-i\omega_k t}) e_k(x)$$
be the solution of \eqref{wavegeneral2}, and note that, by the reasoning above, 
setting $a_k:=\frac{1}{N}\bar a_k$ and $b_k:=\frac{1}{N}\bar b_k$ we have that 
$$u(t,x):=\sum_{k=1}^\infty ( a_k e^{i\omega_k t}+ b_k e^{-i\omega_k t}) 
e_k(x)=\frac{1}{N}\bar u(t,x)$$
is the solution of \eqref{wavegeneral}. Hence to prove that $u(t,x)=\mathcal 
F_\delta \bar u(t,x)$ it 
it suffices to note that by Assumption \ref{A1} (see in particular 
\eqref{fold}) 
\begin{align*}
\mathcal F_\delta \bar u(t,x)&=\sum_{k=1}^\infty (\bar a_k e^{i\omega_k t}+ 
\bar b_k e^{-i\omega_k t}) \mathcal F_\delta 
e_k(x)\\
&=\frac{1}{N}\sum_{k=1}^\infty (\bar a_k e^{i\omega_k t}+ \bar b_k 
e^{-i\omega_k t}) e_k(x)= \frac{1}{N}\bar u(t,x).
\end{align*}
Finally, we show \eqref{sol}: for each $h=1,\dots,N$ we have
\begin{align*}
\bar u(t,x)&=\delta_h^2 \bar u(t,x) =\delta_h \sum_{k=1}^\infty (\bar a_k 
e^{i\omega_k t}+ \bar b_k e^{-i\omega_k t})  \delta _h e_k(x)\\
&=\sum_{k=1}^\infty (\bar a_k e^{i\omega_k t}+ \bar b_k e^{-i\omega_k t}) 
e_k(K_h x)=\delta_h \bar u(t,K_h x)
\end{align*}
and this concludes the proof.
\eqref{wavegeneral2}.
\end{proof}
We are now in position to state the main result of this section, that bridges 
observability of tiles
with their prolongations.
\begin{theorem}\label{thmprol}
  Let $\Omega_1,\Omega_2$ and $\{e_k\}$ satisfy Assumption \ref{A1} and 
Assumption \ref{A2}. 
 Let $u$ be the solution of
 \begin{equation}\label{wavegeneralprol}
 \begin{cases}
  u_{tt}-\Delta u=0& \text{on }\RR\times \Omega_1\\
  u=0&\text{in }\RR\times \partial \Omega_1\\
  u(t,0)=u_0,~u_t(t,0)=u_1&\text{in }\Omega_1
 \end{cases}
\end{equation}
with $u_0,u_1\in D^0(\Omega_1)\times D^{-1}(\Omega_1)$ and let $\overline u$ be 
the solution of
 \begin{equation}\label{wavegeneralprol2}
 \begin{cases}
  u_{tt}-\Delta u=0& \text{on }\RR\times \Omega_2\\
  u=0&\text{in }\RR\times \partial \Omega_2\\
  u(t,0)=\mathcal P_\delta u_0,~u_t(t,0)=\mathcal P_\delta u_1&\text{in 
}\Omega_2.
 \end{cases}
\end{equation}
 Also let $\bar S\subset \Omega_2$ and define 
 $$S:=\bigcup_{h=1}^N K_h^{-1}\overline S \cap \Omega_1.$$
 Then for every $T>0$ and for every couple $c=(c_1,c_2)$ of positive constants, 
the inequalities
 \begin{equation}\label{c1}
\|u_0\|_0^2+\|u_1\|^2_{-1}\asymp_c \int_0^T\int_{S} |u(t,x)|^2 dx. 
\end{equation}
hold if and only if 
 \begin{equation}\label{c2}
\|\mathcal P_\delta u_0\|_0^2+\|\mathcal P_\delta u_1\|^2_{-1}\asymp_c 
\int_0^T\int_{\overline S} |\overline u(t,x)|^2 dx. 
\end{equation}
\end{theorem}
\begin{proof}
 By Lemma \ref{lprol}, $u$ and $\overline u$ satisfy
 $$u(t,x)=\frac{\delta_h}{N}\overline u(t,K_h x)\quad \text{ for all 
}h=1,\dots,N.$$
  Since $\Omega_1$ tiles $\Omega_2$, then setting $S_h:= K_h^{-1}\bar S\cap 
\Omega_1$ we have $S=\bigcup_{h=1}^N S_h$ and $\overline S=\bigcup_{h=1}^N K_h 
S_h$, and that these unions are disjoint. Hence, also recalling 
$|\delta_h|\equiv 1$  we have
 \begin{align*}
\int_I\int_{\overline S} |\overline u(t,x)|^2 dx&= \sum_{h=1}^N\int_I\int_{K_h 
S_h} |\overline u(t,x)|^2 dx\\
&= \sum_{h=1}^N\int_I\int_{S_h} |\overline u(t,K_h x)|^2 dx\\
&= N^2\sum_{h=1}^N\int_I\int_{S_h} |\frac{\delta_h}{N}\overline u(t,K_h x)|^2 
dx\\
&= N^2\sum_{h=1}^N\int_I\int_{S_h} |u(t, x)|^2 dx\\
&= N^2\int_I\int_{S} |u(t, x)|^2 dx\\
\end{align*}
Finally, by Lemma \ref{lcr}  
 $$\|\mathcal P_\delta u_0\|^2_0= N^2\| u_0\|_0^2\quad\text{and}\quad 
\|\mathcal P_\delta u_1\|^2_{-1}=N^2\| u_1\|_{-1}^2.$$
and this implies the equivalence between \eqref{c1} and \eqref{c2}.
\end{proof}

\section{Proof of Theorem \ref{thmmainintro}}\label{st1}
The proof of Theorem \ref{thmmainintro} is based on the application of Theorem 
\ref{thmprol} to the particular case 
$$\Omega_1=\mathcal T \quad \text{and}\quad \Omega_2=\mathcal R.$$
$2/\sqrt{3}$)
as well as
side $2/\sqrt{3}$. 
wave equation in $\mathcal T$ bridge the well-established solutions in the 
rectangle $\mathcal R$ to the ones with domain  equal to the rhombus or the 
hexagon.

We then need to admissibly tile $\mathcal R$ with $\mathcal T$ and a base 
$\{e_k\}$ formed by the eigenfunctions of $-\Delta$ in $H_0^1(\mathcal T)$ 
satisfying Assumption \ref{A2}. Such ingredients are provided in 
\cite{triangle}: in order to introduce them we need some notations. We consider 
the Pauli matrix
$$\sigma_z:=\begin{pmatrix}
             1&0\\
             0&-1\\
            \end{pmatrix}$$
and the rotation matrix 
$$R_\alpha:=\begin{pmatrix}
             \cos \alpha &\sin\alpha\\
             -\sin\alpha&\cos\alpha\\
            \end{pmatrix}$$
where $\alpha:=\pi/3$. Now let $v_1:=(0,1/\sqrt{3})$ and $v_2:=(0,1)$ be two of 
the three vertices of $\mathcal T$ and define the transformations from 
$\RR^2$ onto itself
\begin{equation}\label{kdef}
\small{\begin{array}{ll}
 K_1:=id;\quad &\quad K_4:x\mapsto- R_\alpha(x-v_2)+ 
3v_1\\
 K_2: x\mapsto -R_\alpha \sigma_z(x-v_2)+v_2;&\quad 
K_5:  x\mapsto -R_\alpha(x-v_2)+3 v_1+v_2\\
 K_3: x\mapsto R_\alpha(x-v_2)+v_2;&\quad K_6: x\mapsto 
-x+3 v_1+v_2
\end{array}}
\end{equation}
and note $(\mathcal T, \{K_h\}_{h=1}^6)$ is a tiling for $\mathcal R$. Indeed 
\begin{equation}\label{a1triangle}
cl( \mathcal R)=\bigcup_{h=1}^6 K_h cl(\mathcal T),
\end{equation}
 and the sets $K_h \mathcal T$, for $h=1,\dots,6$, do not overlap -- see Figure 
\ref{folding} and \cite{triangle}. 
 
\begin{figure}
 \includegraphics[scale=0.6]{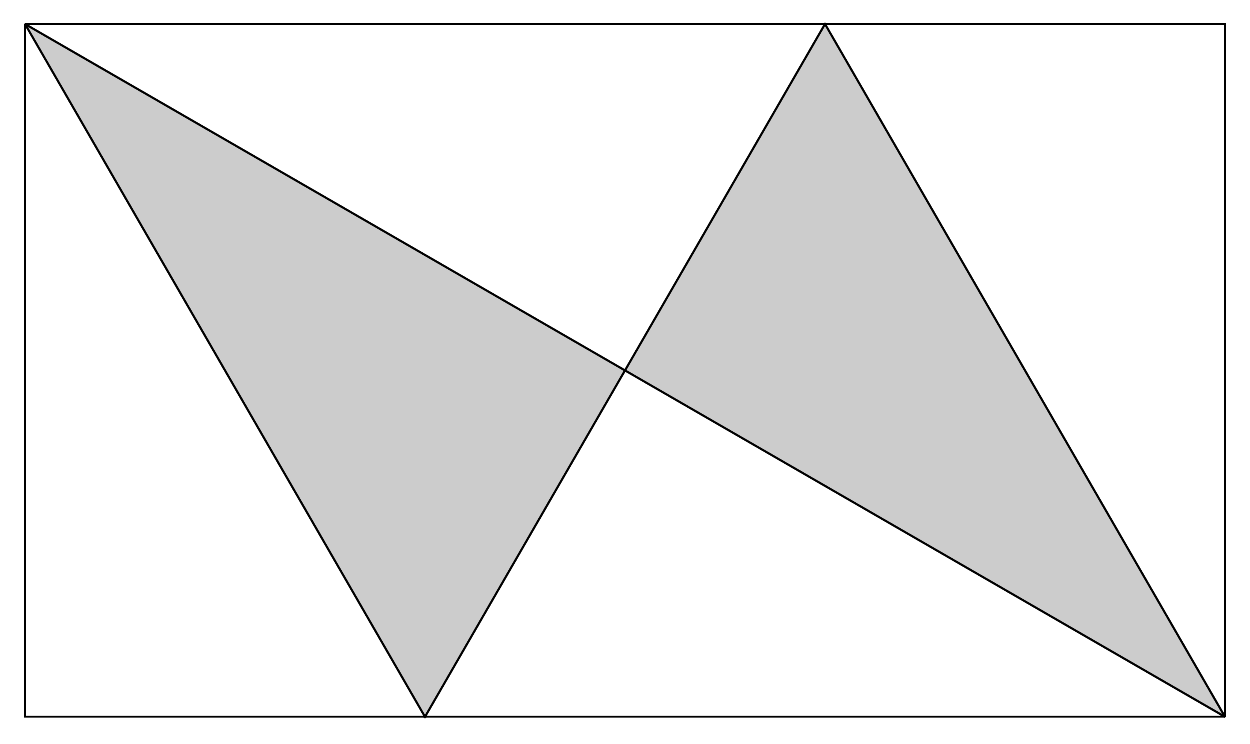}
 \caption{The tiling of $\mathcal R$ with $\mathcal T$, the gray areas 
correspond to negative $\delta_h$'s.  \label{folding}}\end{figure}
We set 
$$\delta:=(1,-1,1,1,-1,1).$$ 
and, in next result, we prove that $\mathcal T$ admissibly tiles $\mathcal R$. 
 
 \begin{lemma}\label{ladm}
  $(\mathcal T,\{K_h\}_{h=1}^6)$ is an admissible tiling of $\mathcal R$.
 \end{lemma}
\begin{proof}
 We want to show that if $\varphi\in H_0^1(\mathcal R)$ then $\mathcal F_\delta 
\varphi\in H_0^1(\mathcal T)$. To this end let $v_0:=(0,0)$, 
$v_1:=(1/\sqrt{3},0)$ and $v_2:=(0,1)$ be the vertices of $\mathcal T$ and 
define 
 $$x_{ij}^\lambda:=\lambda v_i +(1-\lambda) v_{j}.$$
 so that $\partial \mathcal T= \{x_{ij}^\lambda\mid
\lambda\in[0,1], 0\leq i<j\leq2\}$. 
 By a direct computation, for all $\lambda\in[0,1]$ 
 $$K_1(x_{01}^\lambda), K_6(x_{01}^\lambda)\in \partial \mathcal R,$$
 $$K_2(x_{01}^\lambda)=K_4(x_{01}^\lambda),$$
 and
 $$K_3(x_{02}^\lambda)=K_5(x_{02}^\lambda).$$
 Since $\varphi\in H_0^1(\mathcal R)$ then $\mathcal F_\delta 
\varphi(x_{01}^\lambda)=0$.
Similarly, for all $\lambda\in[0,1]$
$$K_1(x_{02}^\lambda), K_6(x_{02}^\lambda)\in \partial \mathcal R,$$
$$K_2(x_{02}^\lambda)=K_3(x_{02}^\lambda),$$
 and
 $$K_4(x_{02}^\lambda)=K_5(x_{02}^\lambda)$$
 therefore $\mathcal F_\delta \varphi(x_{02}^\lambda)=0$ for all 
$\lambda\in[0,1]$.
 Finally for all $\lambda\in[0,1]$ 
$$K_3(x_{12}^\lambda), K_4(x_{12}^\lambda)\in \partial \mathcal R,$$
 $$K_1(x_{12}^\lambda)=K_2(x_{12}^\lambda),$$
 and
 $$K_5(x_{12}^\lambda)=K_6(x_{12}^\lambda)$$
 therefore we get also in this case $\mathcal F_\delta 
\varphi(x_{12}^\lambda)=0$ for all $\lambda\in[0,1]$ and we may conclude that 
$\mathcal F_\delta \varphi\in H_0^1(\mathcal T)$. 
\end{proof}
\begin{remark}
 Lemma \ref{ladm} was remarked in \cite[p.312]{triangle}, but to the 
best of our knowledge, this is the first time an explicit proof is provided. 
\end{remark}

Now, consider the eigenfunctions of $-\Delta$ in $H_0^1(\mathcal R)$:
 $$\overline e_k(x_1,x_2):=\sin(\pi k_1 \frac{x_1}{\sqrt{3}})\sin(\pi k_2 x_2), 
\quad k=(k_1,k_2)\in\NN^2.$$
We finally define for every $k\in \NN^2$
\begin{equation}\label{def}
e_k(x):=N^2 \mathcal F_\delta \overline e_k= \sum_{h=1}^6 \delta_h \overline 
e_k(K_h x). 
\end{equation}

Next result, proved in \cite{triangle}, states that Assumption \ref{A2} is 
satisfied by $\{e_k\}$.
\begin{lemma}
  The set of functions $\{e_k\}$ defined in \eqref{def} is a complete 
orthogonal 
base for $\mathcal T$ 
  formed by the eigenfunction of $-\Delta$ in $H_0^1(\mathcal T)$.   
  Furthermore $\mathcal P_\delta e_k(x)=e_k(x)$.
\end{lemma}

\begin{remark}
 For each $k\in\NN^2$, the eigenfunctions $e_k$ and $\bar e_k$ share the same 
eigenvalue $\gamma_k=\pi^2(\frac{k_1^2}{3}+k_2^2)$, see \cite{triangle}.
\end{remark}
Next gives access to classical results on observability of rectangular 
membranes 
for the study of triangular domains. 
\begin{theorem}\label{thmmain}
 Let $\overline S\subset \mathcal R$ be such that the solution $\overline u$ of
  \begin{equation}\label{wavegeneralrect}
 \begin{cases}
  u_{tt}-\Delta u=0& \text{on }\RR\times \mathcal R\\
  u=0&\text{in }\RR\times \partial \mathcal R\\
  u(t,0)=\mathcal P_\delta u_0,~u(t,0)=\mathcal P_\delta u_1&\text{in }\mathcal 
R.
 \end{cases}
\end{equation}
 satisfies for some $T>0$ and some couple of positive constants $c=(c_1,c_2)$ 
 \begin{equation}\label{obsrectangleg}
 \|\mathcal P_\delta u_0\|_0^2+\|\mathcal P_\delta u_1\|^2_{-1}\asymp_c 
\int_0^T\int_{\overline S} |\overline u(t,x)|^2 dx  
 \end{equation}
for all $( u_0, u_1)\in D^0(\mathcal T)\times D^{-1}(\mathcal T)$. 
Moreover let 
$$S:=\bigcup_{h=1}^6 (K^{-1}_h \overline S\cap \mathcal T)$$
Then the solution $u$ of 
  \begin{equation}\label{wavegeneraltr}
 \begin{cases}
  u_{tt}-\Delta u=0& \text{on }\RR\times \mathcal T\\
  u=0&\text{in }\RR\times \partial \mathcal T\\
  u(t,0)=u_0,~u_t(t,0)=u_1&\text{in }\mathcal T.
 \end{cases}
\end{equation}
satisfies 
\begin{equation}\label{obstrg}
\|u_0\|_0^2+\|u_1\|^2_{-1}\asymp_c \int_0^T\int_{S} |u(t,x)|^2 dx. 
\end{equation}
for all $(u_0,u_1)\in D^0(\mathcal T)\times D^{-1}(\mathcal T)$.
\end{theorem}
\begin{proof}
 Since $\mathcal T,\mathcal R$ and $\{e_k\}$ satisfy Assumption \ref{A1} and 
Assumption \ref{A2}, then the claim follows by a direct application of Theorem 
\ref{thmprol} with $\Omega_1=\mathcal T$ and $\Omega_2=\mathcal R$.
\end{proof}

We conclude this section by showing that Theorem \ref{thmmainintro} is a direct 
consequence of Theorem \ref{thmmain}:
\begin{proof}[Proof of Theorem \ref{thmmainintro}] 
 By Lemma \ref{lcr}, if $(u_0,u_1)\in D^0(\mathcal T)\times D^{-1}(\mathcal T)$ 
then 
 $(\mathcal P_\delta u_0,\mathcal P_\delta u_1)\in D^0(\mathcal R)\times 
D^{-1}(\mathcal R)$. The claim hence follows by Theorem \ref{thmmain}. 
also apply. 
\end{proof}

\section{Proof of Theorem \ref{stripsedgeintro}}\label{st2}
In this Section we keep all the notations used in Section \ref{st1}. So 
that for instance $\{K_h\}$ are the transformations given in \eqref{kdef}  and 
$\delta=(1,-1,1,1,-1,1)$. 
Also recall the definitions
$$S_{\alpha}:=  {\mathcal T}\setminus cl(r_\alpha \mathcal T+ 
(\alpha,\alpha))$$
and
$$\overline 
S_{\alpha}:=[(0,\sqrt{3})\times(0,\alpha)]\cup[(0,\alpha)\times(0,1)]$$

As mentioned in the Introduction, to prove Theorem 
\ref{stripsedgeintro} we need two auxiliary results: Lemma \ref{lobs1}, which 
is an internal observability result on rectangles, and Lemma \ref{lobs2}, which 
is a geometric result characterizing $S_\alpha$. 

We begin by recalling  the following:

\begin{proposition}[{\cite[Proposition 1.1]{KomMia2013}}]\label{km}
 Let $\Omega=(0,\ell_1)\times (0,\ell_2)$, let $J_1\subset 
(0,\ell_1)$ and $J_2\subset (0,\ell_2)$ and define 
$$\overline S:=[J_1\times (0,\ell_2)]\cup [(0,\ell_1)\times J_2]$$
Also define the positive constants 
$$t_1:=\inf_{k\in\NN} \int_{J_1} \sin\left(\frac{\pi k x}{\ell_1}\right)dx, 
\quad t_2:=\inf_{k\in\NN} \int_{J_2} \sin\left(\frac{\pi k x}{\ell_2}\right)dx.
$$
and 
$$m:=\min\left\{\frac{t_1}{\ell_1},\frac{t_2}{\ell_2}\right\}$$
If $T>0$ satisfies the condition 
\begin{equation}\label{obstime}
 \frac{(\ell_1^2+\ell_2^2)(\ell_1^4+\ell_2^4)}{T^2 \ell_1^2 
\ell_2^2}<m
\end{equation}
then the solutions $u$ of \eqref{wave} satisfy the estimate
\begin{equation}\label{inverse}
|u_0|_0^2+|u_1|^2_{-1}\leq_c \int_0^T\int_{\bar S} |u(t,x)|^2dxdt
\end{equation}
for all $(u_0,u_1)\in D^1(\Omega)\times D^0(\Omega)$, with
\begin{equation}\label{c}
c:=\frac{T}{\pi}\left(m
-\frac{(\ell_1^2+\ell_2^2)(\ell_1^4+\ell_2^4)}{T^2 \ell_1^2 
\ell_2^2}\right)>0.
\end{equation}

\end{proposition}

We apply Proposition \ref{km} to prove 
\begin{lemma}\label{lobs1}
Let $T>T_\alpha$. Then the solutions $\overline u$ of  
\eqref{wavegeneralrectintro} satisfy the estimate 
\begin{equation}\label{inverser}
|u_0|_0^2+|u_1|^2_{-1}\leq_{c_\alpha} \int_0^T\int_{\bar S_\alpha} 
|u(t,x)|^2dxdt
\end{equation}
for all $(u_0,u_1)\in D^0(\mathcal R)\times D^{-1}(\mathcal R)$.
\end{lemma}
\begin{proof}
  We apply Proposition \ref{km} to $\Omega=\mathcal R$, so that 
$\ell_1=\sqrt{3}$ and $\ell_2=1$, and to $\overline S=\overline S_\alpha$  -- 
so that $J_1=J_2=(0,\alpha)$. Recall from the statement of Theorem 
\ref{stripsedgeintro} the definition 
$$t_\alpha:=\inf_{k\in \NN} \int_0^\alpha \sin^2(\pi k x/\sqrt{3})dx.$$
Moreover let
$$t'_\alpha:= \inf_{k\in\NN} \int_{0}^\alpha \sin\left(\pi k 
x\right)dx
$$
and
$$m_\alpha:=\min\{t_\alpha/\sqrt{3},t'_\alpha\}.$$
With a little algebraic manipulation of 
\eqref{obstime} we have that if
$$T>\overline T_\alpha:=8\sqrt{\frac{5}{\sqrt{3}}m_\alpha}$$
then 
\begin{equation}
|u_0|_0^2+|u_1|^2_{-1}\leq_{\overline c_\alpha} \int_0^T\int_{\bar S_\alpha} 
|u(t,x)|^2dxdt
\end{equation}
for all $(u_0,u_1)\in D^0(\mathcal R)\times D^{-1}(\mathcal R)$, with
$$\overline c_\alpha:=\frac{T}{\pi}\left(m_\alpha
-\frac{40}{3T^2}\right).$$
Since
\begin{align*}
\frac{t_\alpha}{\sqrt{3}}&=\inf_{k\in\NN}\frac{1}{\sqrt{3}}\int_0^\alpha
\sin^2\left(\frac{\pi k x}{\sqrt{3}}\right)dx \\
&=\inf_{k\in\NN}\int_0^{\alpha/\sqrt{3}}
\sin^2\left(\pi k x\right)dx \leq t'_\alpha,
\end{align*}
then $m_\alpha=\frac{t_\alpha}{\sqrt{3}}$. Hence $T_\alpha=\bar T_\alpha$ and $ 
c_\alpha=\bar c_\alpha$ and this concludes the proof.
\end{proof}

Finally we prove 
\begin{lemma}\label{lobs2}If $\alpha\in(0,1/(3+\sqrt{3})]$ then
 $$S_{\alpha}= \bigcup_{h=1}^6 K_h^{-1}\overline S_{\alpha}\cap \mathcal T.$$
\end{lemma}
\begin{proof} First of all we recall  $cl(\mathcal R)=\bigcup_{h=1}^6 
K_h(cl(\mathcal T))$. Then \\$cl(\mathcal 
T)=\bigcap_{h=1}^6 K_h^{-1}(cl(\mathcal R))$ and 
\begin{equation}
 \mathcal T_\alpha:=r_\alpha cl(\mathcal T)+(\alpha,\alpha)=\bigcap_{h=1}^6 
r_\alpha K^{-1}_h(cl(\mathcal R))+(\alpha,\alpha))
\end{equation}
Also recall  that $K_h's$ are affine maps. Then for 
each $a\in \RR,~b\in\RR^2$
$$ K_h(a x + b) =a K_h(x)+K_h(b)-K_h(0)\quad \forall x\in 
\RR^2,~ h=1,\dots,6.$$
In particular for all $h=1,\dots,6$
\begin{equation}\label{b}
K_h( r_\alpha K_h^{-1}cl(\mathcal R) + (\alpha,\alpha)) = 
 r_\alpha cl(\mathcal R)+K_h(\alpha,\alpha)-r_\alpha K_h(0)
\end{equation}
By a direct 
computation, the assumption $\alpha\geq 0$ implies
\begin{equation}\label{alphasys}
K_h(\alpha,\alpha)-r_\alpha K_h(0)\geq (\alpha,\alpha),\quad \forall h=1,\dots,6
 \end{equation}
where vector inequalities are meant componentwise. Since 
the complement set $\overline S_\alpha^\complement$ of $\overline S_\alpha$ 
satisfies 
$$\{(x_1,x_2)\in \RR^2 \mid x_1,x_2\geq \alpha\}\subset \overline 
S_\alpha^\complement, $$
\eqref{alphasys} and \eqref{b} imply
$$K_h( r_\alpha K_h^{-1}cl(\mathcal R) + (\alpha,\alpha)) \subset 
S_\alpha^\complement \quad \forall h=1,\dots,6.$$
Therefore 
$$r_\alpha K_h^{-1}cl(\mathcal R) + (\alpha,\alpha) \subset 
K^{-1}_h S_\alpha^\complement \quad \forall h=1,\dots,6$$
and, consequently, 
\begin{align*}
 \mathcal T_\alpha
&= \bigcap_{h=1}^6r_\alpha K^{-1}_h( cl(\mathcal R))+ 
(\alpha,\alpha)
%
\subset  
\bigcap_{h=1}^6 K^{-1}_h(\overline S_\alpha^\complement) \end{align*}
By the definition of $S_\alpha$ we then have
\begin{align*}
 S_\alpha&=  \mathcal T_\alpha^\complement\cap \mathcal T \supseteq 
\left(\bigcap_{h=1}^6 K^{-1}_h(\overline 
S_\alpha^\complement)\right)^\complement\cap \mathcal T=  \bigcup_{h=1}^6 
K^{-1}_h
\overline S_\alpha \cap\mathcal T.
\end{align*}
To prove the other inclusion, note that if
$$(x_1,x_2)\in \mathcal 
T_\alpha^\complement=co\{(0,0),(r_\alpha/\sqrt{3},0),(0,r_\alpha)\}
^\complement+(\alpha,\alpha)$$ then either $x_1\leq \alpha$, $x_2\leq \alpha$ 
or $x_2\geq -\sqrt{3}x_1+1-2\alpha$. Hence
$$\mathcal T\setminus \mathcal T_\alpha= A_1\cup A_2$$
with 
\begin{align*}
A_1:=&\{(x_1,x_2)\in \mathcal T \mid x_1\leq \alpha \text{ or } x_2\leq \alpha
\}\\
&=\overline S_\alpha\cap \mathcal T=K_1^{-1}(\overline S_\alpha)\cap 
\mathcal T\subset \bigcup_{h=1}^6K^{-1}_h
\overline S_\alpha \cap\mathcal T
\end{align*}
and
\begin{align*}
A_2&:=\{(x_1,x_2)\in \mathcal T \mid x_2\geq -\sqrt{3}x_1+1-2\alpha
\}\\
&=K_4^{-1}(\overline S_\alpha)\cap 
\mathcal T\subset \bigcup_{h=1}^6K^{-1}_h
\overline S_\alpha \cap\mathcal T.
\end{align*}
\end{proof}

%
%

\end{document}